\documentclass[12pt]{amsart}

\usepackage{amssymb}

\usepackage{enumerate}

\usepackage{graphicx}

\makeatletter
\@namedef{subjclassname@2010}{%
  \textup{2010} Mathematics Subject Classification}
\makeatother

\newtheorem{thm}{Theorem}

\newtheorem{lem}[thm]{Lemma}

\theoremstyle{definition}

\newcommand{\ord}{\operatorname{ord}}
\providecommand{\abs}[1]{\left\lvert#1\right\rvert}

\begin{document}


\baselineskip=17pt



\title[On prime factors of terms of binary recurrence sequences]{On prime factors of terms of binary recurrence sequences}

\author[C.L.~Stewart]{C.L.~Stewart}
\address{Department of Pure Mathematics, University of Waterloo \\
Waterloo, Ontario, Canada N2L 3G1}
\email{cstewart@uwaterloo.ca}

\dedicatory{In memory of Professor Andrzej Schinzel}

\date{}




\maketitle

\section{Introduction}
Let $r$ and $s$ be integers with $r^2+4s\neq 0$. Let $u_0$ and $u_1$ be integers and put
\begin{equation} \label{eq1}
u_n=ru_{n-1}+su_{n-2},
\end{equation}
for $n=2,3,\dots$ . Then for $n\geq0$
\begin{equation} \label{eq2}
u_n=a\alpha^n+b\beta^n,
\end{equation}
where $\alpha$ and $\beta$ are the roots of the characteristic polynomial $x^2-rx-s$ and
\begin{equation} \label{eq5}
a=\frac{u_1-u_0\beta}{\alpha-\beta},\qquad b=\frac{u_0\alpha-u_1}{\alpha-\beta}
\end{equation}
when $\alpha\neq\beta.$ The sequence of integers $(u_n)^\infty_{n=0}$ is a binary recurrence sequence. It is said to be non-degenerate if $ab\alpha\beta\neq0$ and $\alpha/\beta$ is not a root of unity.

In 1934 Mahler \cite{18} proved that if $u_n$ is the $n$-th term of a non-degenerate binary recurrence sequence then the greatest prime factor of $u_n$ tends to infinity with $n$. His proof was ineffective however since it depended on a p-adic version of the Thue-Siegel theorem. In 1967 Schinzel \cite{27} refined work of Gelfond on estimates for linear forms in the logarithms of two algebraic numbers and as a consequence he was able to give an effective lower bound. For any integer $m$ let $P(m)$ denote the greatest prime factor of $m$ with the convention that $P(0)=P(\pm 1)=1.$ Schinzel proved that there exists a positive number $C_0$ which is effectively computable in terms of $a,$ $b,$ $\alpha$ and $\beta$ such that
$$
P(u_n)>C_0n^{c_1}(\log n)^{c_2},
$$
where
$$
(c_1,c_2)=\begin{cases}
(1/84, 7/12) & \text{if}\ \alpha\ \text{and}\ \beta\ \text{are integers} \\
(1/133,7/19) & \text{otherwise}.
\end{cases}
$$
The above result was subsequently improved by Stewart \cite{39}, by Yu and Hung \cite{45} and in 2013 by Stewart \cite{St} who showed that there is a positive number $C$, which is effectively computable in terms of $a,b,\alpha$ and $\beta$ such that if $n$ exceeds $C$ then
\begin{equation} \label{eq7}
P(u_n)>n^{1/2}\exp(\log n/104 \log\log n).
\end{equation}

Let $(t_n)^\infty_{n=0}$ be a non-degenerate binary recurrence sequence with $t_0=0$ and $t_1=1.$ Then, recall \eqref{eq2} and \eqref{eq5},
\begin{equation} \label{eq10}
t_n=\frac{\alpha^n-\beta^n}{\alpha-\beta}
\end{equation}
for $n=0,1,2,\dots\ $ and the sequence is known as a Lucas sequence. (Note that a Lucas sequence is non-degenerate.) Lucas sequences have a rich divisibility structure and have been extensively studied, eg. \cite{6}, \cite{8}, \cite{9}, \cite{14}, \cite{17}, \cite{37} and \cite{46}. In 2013 Stewart \cite{41} proved that if $t_n$ is the $n$-th term of a Lucas sequence then
\begin{equation} \label{eq9}
P(t_n)>n\exp(\log n/104\log\log n)
\end{equation}
provided that $n$ exceeds a number which is effectively computable in terms of $\alpha$ and $\beta$, see also \cite{7a} and \cite{HH}.

In 1967 Schinzel \cite{27} introduced a class of binary recurrence sequences which includes the Lucas sequences and whose members have similar divisibility properties to the Lucas sequences. He considered those sequences for which $a/b$ and $\alpha/\beta$ are multiplicatively dependent and proved that if $\alpha$ and $\beta$ are real numbers then there is a positive number $c$, which is effectively computable in terms of $a,b,\alpha$ and $\beta$, such that
\begin{equation} \label{eq11}
P(u_n)>n-c.
\end{equation}
Schinzel's proof of \eqref{eq11} depended on a result \cite{26} of his on primitive divisors of Lucas numbers. In 2003 Luca \cite{16} proved \eqref{eq11} in the case when $\alpha$ and $\beta$ are not real numbers. Observe that if $(u_n)^\infty_{n=0}$ is a non-degenerate binary recurrence sequence with a term which is zero then $a/b$ and $\alpha/\beta$ are multiplicatively dependent.

 We shall prove the following result.

\begin{thm} \label{Theorem 1}
Let $(u_n)^\infty_{n=0}$ be a non-degenerate binary recurrence sequence, as in  \eqref{eq2}, with $a/b$ and $\alpha/\beta$ multiplicatively dependent. There exists a positive number $C$, which is effectively computable in terms of $a,b,\alpha$ and $\beta$, such that if $n$ exceeds $C$ then
\begin{equation} \label{eq12}
P(u_n)>n\exp(\log n/104\log\log n).
\end{equation}
 
\end{thm}

The proof of Theorem \ref{Theorem 1} relies on arguments from \cite{41} as well as the work of Schinzel \cite{Sch} on primitive divisors in algebraic number fields.

For any non-degenerate binary recurrence sequence $(u_n)^\infty_{n=0}$ we are able to improve \eqref{eq7} for all positive integers $n$ except perhaps for a set of asymptotic density zero. Let $\varepsilon(n)$ be a real valued function on the positive integers for which $\lim_{n\rightarrow\infty}\varepsilon(n)=0.$ In \cite{39} Stewart proved that for all positive integers, except perhaps for a set of asymptotic density zero,
$$
P(u_n)>\varepsilon(n)n\log n;
$$
see the papers of Murty, S\'eguin and Stewart \cite{17} and Balaji and Luca \cite{BL} for related work. Combining the approaches of \cite{39} and \cite{41} we are able to prove the following result.

\begin{thm} \label{Theorem 2}
Let $(u_n)^\infty_{n=0}$ be a non-degenerate binary recurrence sequence. For all positive integers $n$, except perhaps a set of asymptotic density zero,\begin{equation} \label{eq12}
P(u_n)>n\exp(\log n/104\log\log n).
\end{equation}
 
\end{thm}

The proofs of Theorem \ref{Theorem 1} and Theorem \ref{Theorem 2} ultimately depend on an estimate for p-adic linear forms in the logarithms of algebraic numbers due to Yu \cite{Yu} and,
as discussed in \cite{41}, the constant $104$ which appears in our estimates has no arithmetical significance but instead is a consequence of the bounds in \cite{Yu}.
For a more detailed historical account of these topics see \cite{St}.

\section{Cyclotomic polynomials}

For each positive integer $k$ put $\zeta_k=e^{2\pi i/k}$. Let $n$ be a positive integer. The $n$-th cyclotomic polynomial $\Phi_n(x,y)$ is given by
\begin{equation} \label{eq13}
\Phi_n(x,y)=\prod^n_{\substack{j=1 \\ (j,n)=1}}(x-\zeta^j_ny).
\end{equation}
Let $e$ be a positive integer and let $i$ be an integer. Put
\begin{equation} \label{eq14}
\Phi_{n,e}^{(i)}(x,y)=\prod^{ne}_{\substack{j=1 \\ (j,ne)=1\\ j\equiv i \bmod e}}(x-\zeta^j_{ne}y).
\end{equation}
Note that if $(i,e)>1$ then $\Phi_{n,e}^{(i)}(x,y)=1$ and that 
\begin{equation} \label{eq15}
\prod^e_{\substack{i=1 \\ (i,e)=1}}\Phi_{n,e}^{(i)}(x,y)=\Phi_{ne}(x,y).
\end{equation}
We remark that when $(i,e)=1$ the degree of $\Phi_{n,e}^{(i)}(x,y)$ is $\phi(ne)/\phi(e)$ where $\phi( )$ denotes Euler's totient function.

For any integer $i$ we have
\begin{equation} \label{eq16}
\prod^{ne}_{\substack{j=1\\ j\equiv i \bmod e}}(x-\zeta^j_{ne}y)= x^n-\zeta^i_ey^n
\end{equation}
and so by the inclusion-exclusion principle, see also Lemma 4 of \cite{Sch}, when $(i,e)=1$
\begin{equation} \label{eq17}
\Phi_{n,e}^{(i)}(x,y)=\prod_{\substack{m|n \\ (m,e)=1\\ \overline mm\equiv i \bmod e}}(x^{n/m}-\zeta^{\overline m}_{e}y^{n/m})^{\mu(m)}.
\end{equation}
It follows from \eqref{eq17} that $\Phi_{n,e}^{(i)}(x,y)$ has coefficients in $\mathbb Q(\zeta_e)$ and then from \eqref{eq14} that the coefficients of $\Phi_{n,e}^{(i)}(x,y)$ are from $\mathbb Z[\zeta_e]$, the ring of algebraic integers of $\mathbb Q(\zeta_e)$.

Next we put
\begin{equation} \label{eq18}
\Psi_{n,e}^{(i)}(x,y)=\prod^{ne}_{\substack{j=1 \\ (j,ne)>1\\ j\equiv i \bmod e}}(x-\zeta^j_{ne}y).
\end{equation}
By \eqref{eq16} we have
\begin{equation} \label{eq119}
\Phi_{n,e}^{(i)}(x,y)\Psi_{n,e}^{(i)}(x,y)=x^n-\zeta^i_{e}y^n.
\end{equation}
Since $\Phi_{n,e}^{(i)}(x,y)$ is in $\mathbb Z[\zeta_e][x,y]$ we see from \eqref{eq18} and \eqref{eq119} that $\Psi_{n,e}^{(i)}(x,y)$ is also in $\mathbb Z[\zeta_e][x,y]$.

\section{Divisibility of values of the cyclotomic polynomial and of Lucas numbers}

We first record two results describing the arithmetical character of values of the cyclotomic polynomial.
\begin{lem} \label{lem1}
Suppose that $(\alpha+\beta)^2$ and $\alpha\beta$ are coprime non-zero integers and that $\alpha/\beta$ is not a root of unity. If $n>4$ and $n\neq 6,12$ then $P(n/(3,n))$ divides $\Phi_n(\alpha,\beta)$ to at most the first power. All other prime factors of $\Phi_n(\alpha,\beta)$ are congruent to $\pm 1\pmod{n}.$
\end{lem}

\begin{proof}
This is Lemma 6 of \cite{37}. 
\end{proof}

Our next result follows from the proof of Theorem 1.1 of \cite{41}. Note that we do not require $(\alpha+\beta)^2$ and $\alpha\beta$ to be coprime.
\begin{lem} \label{lem 2}
Let $\alpha$ and $\beta$ be complex numbers such that $(\alpha+\beta)^2$ and $\alpha\beta$ are non-zero integers and $\alpha/\beta$ is not a root of unity. There exists a positive number $C,$ which is effectively computable in terms of $\alpha$ and $\beta$, such that for $n>C,$
\begin{equation} \label{eq19}
P(\Phi_n(\alpha,\beta))>n\exp(\log n/103.95\log\log n).
\end{equation}
\end{lem}

\begin{proof}
This follows from the second last line in the proof of Theorem 1.1 of \cite{41}.
\end{proof}

For any non-zero rational number $x$ let $\ord_p x$ denote the p-adic order of $x$.

\begin{lem}  \label{lem3}
Let $(u_n)^\infty_{n=0}$ be a non-degenerate binary recurrence sequence as in \eqref{eq2} with $a/b$ and $\alpha/\beta$ multiplicatively independent. There exists a positive number $C$ which is effectively computable in terms of $a,$ $b,$ $\alpha$ and $\beta$ such that if $p$ exceeds $C$ then
$$
\ord_pu_n<p\exp(-\log p/51.9\log\log p)\log n.
$$
\end{lem}

\begin{proof}
This is Lemma 7 of \cite{St}.
\end{proof}

We shall now describe the prime decomposition of terms of a Lucas sequence $(t_n)^\infty_{n=0}$ .

\begin{lem} \label{lem4} 
Let $(t_n)^\infty_{n=0}$ be a Lucas sequence as in \eqref{eq10}. If $p$ is a prime number which does not divide $\alpha\beta$ then $p$ divides $t_n$ for some positive integer $n$ and if $l$ is the smallest positive integer for which $p$ divides $t_l$ then

$$
l\leq p+1.
$$
\end{lem}

\begin{proof}
This follows, for example, from Lemma 7 of \cite{39}.
\end{proof}

For any rational number $x$ let $\abs{x}_p$ denote the p-adic value of $x$, normalized so that $\abs{p}_p = p^{-1}.$
\begin{lem}\label{prop:2}
Let $\left\{t_n\right\}_{n=0}^{\infty}$ be a Lucas sequence, as in \eqref{eq10}, with $\alpha +\beta$ and $\alpha\beta$ coprime. Let $p$ be a prime number which does not divide $\alpha\beta$, let $\ell$ be the smallest positive integer for which $p$ divides $t_\ell$ and let $n$ be a positive integer. If $\ell$ does not divide $n$, then
\begin{align*}
|t_n|_p=1.
\end{align*}
If $n= \ell k$ for some positive integer $k$, we have, for $p > 2$,
\begin{align*}
\abs{t_n}_p &= \abs{t_\ell}_p \abs{k}_p,
\intertext{while for $p=2$,}
\abs{t_n}_2 &= \begin{cases}\abs{t_\ell}_2 & \text{ for } k \text{ odd }\\ 2\abs{t_{2\ell}}_2 \abs{k}_2 & \text{ for } k \text{ even. }\end{cases}
\end{align*}
\end{lem}
\begin{proof}
This is Lemma 8 of \cite{39} .
\end{proof}

\begin{lem}\label{lem8}
Let $\left\{t_n\right\}_{n=0}^{\infty}$ be a Lucas sequence, as in \eqref{eq10}, with $\alpha +\beta$ and $\alpha\beta$ coprime and $|\alpha| \geq |\beta|$. Let $n$ be an integer larger than $1$. There exists a positive number $C$, which is effectively computable in terms of $\alpha$ and $\beta$, such that if $p$ is a prime number larger than $C$ then
$$
\ord_pt_n<p\exp(-\log p/51.9\log\log p)\log |\alpha|\log n.
$$

\end{lem}
\begin{proof}
We may suppose that $C$ exceeds $|\alpha\beta|$ and the absolute value of the discriminant of $\mathbb Q(\alpha/\beta)$. The result then follows from Lemma 4.3 of \cite{41}.
\end{proof}

\section{Cyclotomic polynomials at algebraic points in quadratic cyclotomic extensions}

Let $\theta_1$ and $\theta_2$ be non-zero algebraic integers in $\mathbb Q(\zeta_e)$ with $e$ equal to $3,4$ or $6$ and suppose that $\theta_1/\theta_2$ is not a root of unity and that $\theta_1=\overline\theta_2$. Then $\theta_1$ and $\theta_2$ are algebraic conjugates. Put
$$
g=((\theta_1+\theta_2)^2, \theta_1\theta_2),
$$
and
$$
\lambda_1= \theta_1/\sqrt g     ,      \lambda_2= \theta_2/\sqrt g. 
$$
Note that
$$
(x-\lambda_1)(x+\lambda_1)(x-\lambda_2)(x+\lambda_2) = x^4-((\theta_1+\theta_2)^2/g-2\theta_1\theta_2/g)x^2-(\theta_1\theta_2/g)^2
$$
is a polynomial with integer coefficients and thus $\lambda_1$ and $\lambda_2$ are algebraic integers. Further $\lambda_1$ is of degree $2$ over $\mathbb Q$ with conjugate $\lambda_2$ when $g$ is a perfect square and is of degree $4$ over $\mathbb Q$ with conjugates $\lambda_1,-\lambda_1,\lambda_2,-\lambda_2$ when $g$ is not a perfect square.
Since $\theta_1/\theta_2=\lambda_1/\lambda_2$ is not a root of unity we see that $\lambda_1$ is not a root of unity. In both cases the conjugates of $\lambda_1$ have the same absolute value as $\lambda_1$ and, since $\lambda_1$ is not a root of unity,
\begin{equation} \label{eq20}
|\lambda_1|\geq 2^{1/4},
\end{equation}
as is readily checked. Furthermore, since $\theta_1=\overline\theta_2$ we find that $\overline{\lambda_1/\lambda_2} =\lambda_2/\lambda_1$ and as $\lambda_1/\lambda_2$ is not a root of unity it is an algebraic number of degree at least $2$. In fact it has conjugate $\lambda_2/\lambda_1$ and minimal polynomial
$$
\lambda_1\lambda_2x^2-(\lambda_1^2+\lambda_2^2)x-\lambda_1\lambda_2.
$$
For any algebraic number $\alpha$ let $M(\alpha)$ denote the Mahler measure of $\alpha$, see \cite{3a}. We then have
\begin{equation} \label{eq21}
M(\lambda_1/\lambda_2)=M(\lambda_2/\lambda_1)=|\lambda_1\lambda_2|\max(1,|\lambda_1/\lambda_2|)\max(1,|\lambda_2/\lambda_1|)=|\lambda_1|^2.
\end{equation}

\begin{lem}\label{lem9}
Let $n$ be a positive integer and $\zeta$ an $e$-th root of unity with $e$ equal to $3,4$ or $6$. There exists an effectively computable positive number $c_1$ such that

$$
n\log|\lambda_1|-c_1\log(n+1)\log|\lambda_1| \leq \log|\lambda_1^n-\zeta\lambda_2^n| \leq n\log|\lambda_1| + \log 2.
$$

\end{lem}
\begin{proof}
Note that
$$
\log|\lambda_1^n-\zeta \lambda_2^n|=n\log|\lambda_1|+\log|\zeta(\lambda_2/\lambda_1)^n-1|
$$
Since $\theta_1=\overline\theta_2$ we see that $|\lambda_2/\lambda_1|=1$ and so $|\zeta(\lambda_2/\lambda_1)^n-1| \leq 2$. It remains to establish a lower bound for $|\zeta(\lambda_2/\lambda_1)^n-1|$. For any complex number $z$, either $1/4\leq|e^{z}-1|$ or
$$
|z-ib\pi| \leq 4|e^{z}-1|
$$
for some integer $b$, see page 176 of \cite{5}. Let $z=\log\zeta + n\log \lambda_2/\lambda_1$ where we take the principal value of the logarithms. Then either
\begin{equation} \label{eq22}
|\zeta(\lambda_2/\lambda_1)^n-1| \geq 1/4
\end{equation}
or
$$
4|\zeta(\lambda_2/\lambda_1)^n-1| \geq \min_{b\in \mathbb Z}|\log \zeta +n\log (\lambda_2/\lambda_1)-b\pi i|.
$$
Suppose that the minimum occurs at $b_0$. Then $|b_0| \leq n+1.$ Further
$$
\log \zeta - b_0\pi i = b_1\log \zeta_{12}
$$
with $|b_1| \leq 6(|b_0|+1) \leq 6n+12$ and thus if \eqref{eq22} does not hold then
\begin{equation} \label{eq23}
4|\zeta(\lambda_2/\lambda_1)^n-1| \geq |n\log (\lambda_2/\lambda_1) + b_1\log \zeta_{12}|.
\end{equation}
Let $c_1,c_2,\dots$ denote effectively computable positive numbers. This is a linear form in two logarithms and by \cite{LMN}, \cite{BW} or \cite{20} we see from \eqref{eq22} and \eqref{eq23}, since $\lambda_2/\lambda_1$ is not a root of unity, that
\begin{equation} \label{eq24}
\log |\zeta(\lambda_2/\lambda_1)^n-1| > -c_2\log(n+1)\log\max(4,A)
\end{equation}
where $A$ is the Mahler measure of $\lambda_2/\lambda_1.$
Thus, by \eqref{eq20} and \eqref{eq21},
$$
\max(4,A) \leq |\lambda_1|^{c_3}
$$
hence, from \eqref{eq24},
$$
\log |\zeta(\lambda_2/\lambda_1)^n-1| > -c_4\log(n+1)\log|\lambda_1|
$$
and our result follows.

\end{proof}
For any positive integer $n$ let $\omega(n)$ denote the number of distinct prime factors of $n$ and put $q(n)=2^{\omega(n)}$.

\begin{lem}\label{lem10}
Let $e$ be $3,4$ or $6$ and let $i$ be an integer coprime with $e$. There exists an effectively computable positive number $c$ such that if $n>2$ then
$$
(\phi(ne)/\phi(e)-cq(n)\log n)\log |\lambda_1| \leq \log |\Phi_{n,e}^{(i)}(\lambda_1,\lambda_2)| \leq (\phi(ne)/\phi(e)+cq(n)\log n)\log |\lambda_1|.
$$

\end{lem}
\begin{proof}
By \eqref{eq17}
$$
\log |\Phi_{n,e}^{(i)}(\lambda_1,\lambda_2)|=\sum_{\substack{m|n \\ (m,e)=1\\ \overline mm\equiv i \bmod e}}\mu(m)\log |\lambda_1^{n/m}-\zeta^{\overline m}_{e}\lambda_2^{n/m}|
$$
and so, by Lemma \ref{lem9},
$$
|\log |\Phi_{n,e}^{(i)}(\lambda_1,\lambda_2)|-\sum_{\substack{m|n \\ (m,e)=1}}\mu(m)(n/m)\log |\lambda_1|| \leq \sum_{\substack{m|n \\ (m,e)=1\\ \mu(m) \neq 0}}c_1\log (n+1)\log |\lambda_1|.
$$ 
The result now follows.
\end{proof}

\begin{lem}\label{lem11}
Let $e$ be $3,4$ or $6$ and let $i$ be an integer coprime with $e$. There exists an effectively computable positive number $C$ such that if $n$ exceeds $C$ then
$$
\log |\Phi_{n,e}^{(i)}(\lambda_1,\lambda_2)| > (\phi(ne)/2\phi(e))\log |\lambda_1|.
$$
\end{lem}

\begin{proof}
For $n$ sufficiently large
$$
\phi(n) > n/2\log \log n $$
and
$$
 q(n)< n^{1/\log \log n}
$$
and so by \eqref{eq20} the result follows from Lemma \ref{lem10}.

\end{proof}

\begin{lem}\label{lem12}
Let $e$ be $3,4$ or $6$ and let $p$ be a prime number. There exists a positive number $C$, which is effectively computable in terms of $a,b,\alpha$ and $\beta$, such that for $n>C$
$$
\ord_p\Phi_{ne}(\lambda_1,\lambda_2)<p\exp(-\log p/51.9\log\log p)\log |\lambda_1|\log ne.
$$
\end{lem}

\begin{proof}
This follows from (5.3) and (5.4) of \cite{41}.
\end{proof}

\begin{lem}\label{lem13}
Let $e$ be $3,4$ or $6$ and let $i$ be $1$ or $-1$. There exists a positive number $C$, which is effectively computable in terms of $\theta_1$ and $\theta_2$, such that if $m$ exceeds $C$ then there is an irreducible $\pi$ in $\mathbb Z[\zeta_e]$ which divides
$$
\theta_1^m-\zeta_e^i\theta_2^m
$$
in $\mathbb Z[\zeta_e]$ which is either a rational prime $p$ or is such that $\pi\overline\pi = p$ and, in both cases,
$$
p> m\exp(\log m/103.95\log \log m).
$$
\end{lem}

\begin{proof}
Let $c_1,c_2,\dots$ denote positive numbers which are effectively computable in terms of $\theta_1$ and $\theta_2$. From Section 2 we see that $\Phi_{m,e}^{(i)}(x,y)$ is a polynomial with coefficients in $\mathbb Z[\zeta_e]$. Thus $\Phi_{m,e}^{(i)}(\theta_1,\theta_2)$ is in $\mathbb Z[\zeta_e]$ and, by \eqref{eq119}, $\Phi_{m,e}^{(i)}(\theta_1,\theta_2)$ divides $\theta_1^m-\zeta_e^i\theta_2^m$ in $\mathbb Z[\zeta_e]$. By \eqref{eq15}
$$
\Phi_{m,e}^{(1)}(\theta_1,\theta_2)\Phi_{m,e}^{(-1)}(\theta_1,\theta_2)=\Phi_{me}(\theta_1,\theta_2)
$$
and therefore
\begin{equation} \label{eq25}
\Phi_{m,e}^{(1)}(\lambda_1,\lambda_2)\Phi_{m,e}^{(-1)}(\lambda_1,\lambda_2)=\Phi_{me}(\lambda_1,\lambda_2).
\end{equation}
Notice that $\Phi_{m,e}^{(j)}(\lambda_1,\lambda_2)=g^{-\phi(me)/2\phi(e)}\Phi_{m,e}^{(j)}(\theta_1,\theta_2)$ for $j=\pm 1.$
Since $\Phi_{m,e}^{(j)}(\theta_1,\theta_2)$ is in $\mathbb Z[\zeta_e]$ and $\Phi_{m,e}^{(j)}(\lambda_1,\lambda_2)$ is an algebraic integer we see that $\Phi_{m,e}^{(j)}(\lambda_1,\lambda_2)$ is in $\mathbb Z[\zeta_e]$ for $j=\pm 1.$ Therefore if $\pi$ is an irreducible in  $\mathbb Z[\zeta_e]$ which divides $\Phi_{m,e}^{(i)}(\lambda_1,\lambda_2)$ then $\pi$ divides $\Phi_{m,e}^{(i)}(\theta_1,\theta_2)$ and so divides  $\theta_1^m-\zeta_e^i\theta_2^m$. We shall now show that $\Phi_{m,e}^{(i)}(\lambda_1,\lambda_2)$ is divisible by an irreducible $\pi$ which is either a large rational prime or is such that $\pi\overline\pi$ is a large rational prime.

Since $(\lambda_1 + \lambda_2)^2$ and $\lambda_1\lambda_2$ are coprime integers $\Phi_{me}(\lambda_1,\lambda_2)$ is an integer for $me>12$ and by Lemma \ref{lem1}  $P(me/(3,me))$ divides $\Phi_{me}(\lambda_1,\lambda_2)$ to at most the first power. All other prime factors are congruent to $\pm 1 (\bmod me)$. Thus
$$
\Phi_{m,e}^{(i)}(\lambda_1,\lambda_2)=\gamma\pi_1^{l_1}\dots \pi_t^{l_t}
$$
where $\gamma$ is a divisor of $P(me/(3,me))$, $t\geq 0$, $\pi_1,\dots , \pi_t$ are irreducibles of $\mathbb Z[\zeta_e]$  and $l_1,\dots,l_t$ are positive integers. Note that $t\geq 1$ for $m>c_1$ by Lemma \ref{lem11}. Let $P$ be the largest prime associated with an irreducible $\pi_j$. Then by \eqref{eq25} and  Lemma \ref{lem12}
$$
\max_j l_j \leq 2P\exp(-\log P/51.9\log\log P)\log |\lambda_1|\log me.
$$
hence
\begin{equation} \label{eq26}
\log |\Phi_{m,e}^{(i)}(\lambda_1,\lambda_2)| \leq \log me +2tP\log P\exp(-\log P/51.9\log\log P)\log |\lambda_1|\log me.
\end{equation}
But $t\leq 2(\pi(P,me,1)+\pi(P,me,-1))$ and so 
\begin{equation} \label{eq27}
t \leq 5P/me. 
\end{equation}
Thus by \eqref{eq26}  and \eqref{eq27}
\begin{equation} \label{eq28}
\log |\Phi_{m,e}^{(i)}(\lambda_1,\lambda_2)| \leq c_2(P^2\log P\exp(-\log P/51.9\log\log P)\log me)/me,
\end{equation}
and by Lemma \ref{lem11}, for $m>c_3$,
\begin{equation} \label{eq29}
\log |\Phi_{m,e}^{(i)}(\lambda_1,\lambda_2)| > (\phi(me)/2\phi(e))\log |\lambda_1|.
\end{equation}
Comparing \eqref{eq28} and \eqref{eq29} we find that, for $m> c_4$,
$$
me\phi(me)/\log m < c_5P^2 \log P\exp(-\log P/51.9\log\log P).
$$
Since $ \phi(me)> c_6m/\log\log m$
$$
P > m\exp(\log m/ 103.95\log \log m)
$$
for $m>c_7$ as required.

\end{proof}

\section{Proof of Theorem 1}

Put $\mathbb K =\mathbb Q(\alpha)$ and let $\mathcal O_\mathbb K$ denote the ring of algebraic integers of $\mathbb K$. Let $w$ be the smallest positive integer for which $wa$ and $wb$ are algebraic integers. By considering the sequence
$(v_n)^\infty_{n=0}$ with $v_n=wu_n$ for $n=0,1,\dots$ we see that it suffices to prove our result for sequences $(u_n)^\infty_{n=0}$ for which $a,b,\alpha$ and $\beta$ are algebraic integers. Since $a/b$ and $\alpha/\beta$ are multiplicatively dependent there exist integers $k$ and $l$, not both zero, for which
\begin{equation} \label{eq30}
(a/b)^k=(\alpha/\beta)^l.
\end{equation}
By inverting \eqref{eq30} if necessary we may suppose that $k\geq 0$. Notice that $k\neq 0$ since otherwise $\alpha/\beta$ is a root of unity contrary to the assumption that $(u_n)^\infty_{n=0}$ is non-degenerate. Thus $k>0$.

If $l=0$ then $a/b$ is a root of unity and we put
\begin{equation} \label{eq31}
u_n=a(\theta_1^n-\zeta\theta_2^n)
\end{equation}
where
$$
(\theta_1,\theta_2)=(\alpha,\beta)
$$
and $\zeta$ is a root of unity from $\mathbb K$.

We now suppose that $k>0$ and $l\neq0$ and, following Schinzel \cite{27} and Luca \cite{16}, we put
$$
l_1=l/(k,l), k_1=k/(k,l).
$$
It follows from \eqref{eq30} that
\begin{equation} \label{eq32}
(a/b)^{k_1}=(\alpha/\beta)^{l_1}\zeta
\end{equation}
where $\zeta$ is a root of unity from $\mathbb K$. There exists a unique pair of integers $(x,y)$ for which
\begin{equation} \label{eq33}
xl_1+yk_1=1
\end{equation}
and
$$
0<y\leq |l_1|.
$$
Put
$$
\rho=a^x\alpha^y/b^x\beta^y.
$$
Then, by \eqref{eq33},
\begin{equation} \label{eq34}
\rho^{l_1}=(a/b)^{xl_1}(\alpha/\beta)^{yl_1}=(a/b)^{xl_1}(a/b)^{yk_1}\zeta^{-y}=(a/b)\zeta^{-y}
\end{equation}
and
\begin{equation} \label{eq35}
\rho^{k_1}=(a/b)^{xk_1}(\alpha/\beta)^{yk_1}=(\alpha/\beta)^{xl_1}\zeta^x(\alpha/\beta)^{yk_1}=(\alpha/\beta)\zeta^x.
\end{equation}
Thus
$$
(a/b)(\alpha/\beta)^n=\rho^{l_1}\zeta^{y}\rho^{k_1n}\zeta^{-xn}=\rho^{l_1+k_1n}\zeta^{y-xn}.
$$
Accordingly
$$
u_n=b\beta^n((a/b)(\alpha/\beta)^n+1)
$$
so
$$
u_n=b\beta^n\zeta^{y-xn}(\rho^{l_1+k_1n}+\zeta^{xn-y}),
$$
and we find that
\begin{equation} \label{eq36}
\theta_2^{l_1+k_1n}u_n=b\beta^n\zeta^{y-xn}(\theta_1^{l_1+k_1n}-(-\zeta^{xn-y})\theta_2^{l_1+k_1n})
\end{equation}
where 
\begin{equation} \label{eq37}
(\theta_1,\theta_2)=(a^x\alpha^y,b^x\beta^y) 
\end{equation}
when $x\geq 0$ and
\begin{equation} \label{eq38}
(\theta_1,\theta_2)=(b^{-x}\alpha^y,a^{-x}\beta^y) .
\end{equation}
when $x<0$. Observe that
$$
\theta_1/\theta_2= \alpha/\beta
$$
in case \eqref{eq31} while
$$
\theta_1/\theta_2= \rho
$$
in cases \eqref{eq37} and \eqref{eq38}. Thus, by \eqref{eq35} and the fact that $\alpha/\beta$ is not a root of unity we see that in all three cases $\theta_1/\theta_2$ is not a root of unity. Furthermore either $a,b,\alpha,\beta$ are non-zero integers or $\alpha$ and $\beta$ are algebraic conjugates hence $\theta_1$ and $\theta_2$ are algebraic conjugates. In both cases $\theta_1+\theta_2$ and $\theta_1\theta_2$ are non-zero integers. Note that in the former case $\mathbb K=\mathbb Q$ and so the root of unity $\zeta$ in \eqref{eq31}, and also in \eqref{eq32}, is $1$ or $-1$.

If in \eqref{eq31} $\zeta$ is $1$ then $\Phi_n(\theta_1,\theta_2)$ divides $u_n$ while if $\zeta$ is $-1$ then $\Phi_{2n}(\theta_1,\theta_2)$ divides $u_n$ and in both cases the result follows from Lemma \ref{lem 2}. If $l\neq 0$ then \eqref{eq36} holds and $\theta_2^{l_1+k_1n}u_n$ is an algebraic integer in $\mathbb K$ which is divisible by $\Phi_{k_1n+l_1}(\theta_1,\theta_2)$ in $\mathcal O_\mathbb K$ if $-\zeta^{xn-y}$ is $1$ and is divisible by $\Phi_{2(k_1n+l_1)}(\theta_1,\theta_2)$ in $\mathcal O_\mathbb K$ if $-\zeta^{xn-y}$ is $-1$. Again the result follows from Lemma \ref{lem 2}.

It remains to consider the possibility that $\zeta$ in \eqref{eq31} or $-\zeta^{xn-y}$ in \eqref{eq36} is a root of unity in $\mathbb K$ different from $1$ or $-1$. Since the degree of $\mathbb K$ is at most $2$ over $\mathbb Q$ we find that the root of unity must be a primitive third, fourth or sixth root of unity and so $\mathbb K = \mathbb Q(\zeta_e)$ with $e$ equal to $3,4$ or $6$, hence equal to $3$ or $4$. Notice that $\mathbb Z [\zeta_e]$ is the ring of algebraic integers of $\mathbb Q(\zeta_e)$and that $\mathbb Z [\zeta_e]$ is a unique factorization domain. Since $\mathbb Q(\alpha)=\mathbb Q(\zeta_e)$ we see that $\alpha$ and $\beta$ and also $a$ and $b$ are complex conjugates hence
$$
\theta_1=\overline\theta_2,
$$
 in all three cases.

Let $c_1,c_2,\dots$ denote positive numbers which are effectively computable in terms of $a,b,\alpha$ and $\beta$. Note that if $\pi$ is an irreducible in $\mathbb Z [\zeta_e]$ which is not a rational prime then $\pi\overline\pi$ is a prime $p$ and since $u_n$ is an integer if $\pi$ divides $u_n$ then $p$ divides $u_n$.        If $\zeta$ in \eqref{eq31} is a root of unity different from $1$ or $-1$ then we may apply Lemma \ref{lem13} with $m$ equal to $n$ to give the result. If $-\zeta^{xn-y}$ in \eqref{eq36} is a root of unity different from $1$ or $-1$ then we may apply Lemma \ref{lem13} with $m=l_1+k_1n$. Since $l_1+k_1n> n/2$ for $n>c_1$ we see that
$$
\theta_1^{l_1+k_1n}-(-\zeta^{xn-y})\theta_2^{l_1+k_1n}
$$
is divisible by an irreducible $\pi$ in $\mathbb Z [\zeta_e]$ which is either a rational prime $p$ or is such that $\pi\overline\pi=p$ and in both cases
$$
p>n\exp(\log n/104\log \log n)
$$
for $n>c_2.$ By \eqref{eq36} $p$ divides $u_n$ since $b\beta^n\zeta^{y-xn}$ is an algebraic integer and for $n>c_3$ we see that neither $\pi$ nor $\overline\pi$ divides $\theta_2$. The result now follows.

\section{Proof of Theorem 2}

Let $u_n$ denote the $n$-th term of  a non-degenerate binary recurrence sequence as in \eqref{eq1} and let $g=(r^2,s)$. Let $\mathbb K = \mathbb Q(\alpha)$ and let $\mathcal O_\mathbb K$ denote the ring of algebraic integers of $\mathbb K$. For any $\theta$ in $\mathcal O_\mathbb K$ let $[\theta]$ denote the ideal in $\mathcal O_\mathbb K$ generated by $\theta$. Notice that
$$
(x-\alpha^2/g)(x-\beta^2/g)=x^2-((r^2+2s)/g)x+(s/g)^2.
$$
Since $(r^2+2s)/g$ and $s/g$ are coprime
$$
([\alpha^2/g],[\beta^2/g])=[1].
$$
Put
$$
v_n=g^{-n}u_{2n}=a(\alpha^2/g)^n+b(\beta^2/g)^n
$$
and
$$
w_n=g^{-n}u_{2n+1}=a\alpha(\alpha^2/g)^n+b\beta(\beta^2/g)^n
$$
for $n=0,1,2,\dots$ .

We shall prove that if $(u_n)^\infty_{n=0}$ is a non-degenerate binary recurrence sequence as in \eqref{eq1} with $([\alpha],[\beta])=[1]$ then for all positive integers $n$, except perhaps a set of asymptotic density $0$,
\begin{equation} \label{eq39}
P(u_n) \geq n\exp(\log n/103.95\log \log n).
\end{equation}
Since $(n/2)-1 \geq n/3$ for $n\geq 6$ and
$$
(n/3)\exp(\log (n/3)/103.95\log \log (n/3))> n\exp(\log n/104 \log \log n)
$$ 
for $n$ sufficiently large we see that this suffices to prove our result in general on considering the non-degenerate binary recurrence sequences $(v_n)^\infty_{n=0}$ and $(w_n)^\infty_{n=0}$ in place of $(u_n)^\infty_{n=0}$ .

Let $c_1,c_2,\dots$ denote positive numbers which are effectively computable in terms of $a,b,\alpha$ and $\beta$. By Theorem \ref{Theorem 1} it suffices to prove our result under the additional assumption that $a/b$ and $\alpha/\beta$ are multiplicatively independent. Further we may assume, without loss of generality, that $|\alpha| \geq |\beta|.$

To establish \eqref{eq39} we shall assume that there is a positive number $\delta$ such that
\begin{equation} \label{eq40}
P(u_m) < m\exp(\log m/103.95\log \log m),
\end{equation}
for a set of integers $m$ of positive upper density $\delta$ and we shall show that this leads to a contradiction.
Accordingly, we can find arbitrarily large integers $n$ such that between $n$ and $2n$ there are at least $\delta n/2$ integers $m$ which satisfy \eqref{eq40}. Fix such an integer $n$ and denote the set of these integers by $M$. Put
\begin{equation} \label{eq41}
T= 2n\exp(\log 2n/103.95\log \log 2n),
\end{equation}
and for each prime number $p$ less than $T$ let $u_{m(p)}$ be the term with $n\leq m(p)\leq 2n$ which is divisible by the highest power of $p$; if more than one term is divisible by $p$ raised to the largest exponent then denote the one with least index by $u_{m(p)}$.

It is proved on page 24 of \cite{39} that, for $n$ sufficiently large, at most $3$ of the integers $m$ with $n\leq m \leq 2n$ satisfy
$$
|u_m| < |\alpha|^{3m/4}.
$$
Further, since $u_m$ is non-zero for $m$ sufficiently large, we see that
\begin{equation} \label{eq42}
\log |\prod_{\substack{m\in M}}u_m| > \frac{\delta n^2}{4}\log |\alpha|
\end{equation}
for $n$ sufficiently large.

Put
$$
S(p)=\frac{u_n\dots u_{2n}}{u_{m(p)}}.
$$
Clearly
\begin{equation} \label{eq43}
 |\prod_{\substack{m\in M}}u_m| \leq \prod_{p<T}|u_{m(p)}|^{-1}_p|S(p)|^{-1}_p.
\end{equation}
By Lemma \ref{lem3}, for $p>c_1$
\begin{equation} \label{eq44}
\log |u_{m(p)}|^{-1}_p < p\log p\exp(-\log p/51.9\log \log p)\log 2n.
\end{equation}
Further, for $p\leq c_1$
\begin{equation} \label{eq45}
\log |u_{m(p)}|^{-1}_p < \max_{n\leq m\leq 2n} \log |u_m| < 4n\log |\alpha|
\end{equation}
for $n$ sufficiently large.
Thus
$$
\sum_{p<T}\log |u_{m(p)}|^{-1}_p \leq \sum_{p\leq c_1}\log |u_{m(p)}|^{-1}_p+ \sum_{c_1<p<T}\log |u_{m(p)}|^{-1}_p
$$
and by \eqref{eq44} and \eqref{eq45}
$$
\sum_{p<T}\log |u_{m(p)}|^{-1}_p \leq c_2n+\pi(T)T\log T \exp(-\log T/51.9\log \log T)\log 2n.
$$
Therefore, by \eqref{eq41}, for $n$ sufficiently large
\begin{equation} \label{eq46}
\sum_{p<T}\log |u_{m(p)}|^{-1}_p < n^2\exp(-\log n/ 40,000\log \log n).
\end{equation}

It remains to estimate $\prod_{p<T}|S(p)|^{-1}_p$.

 Let $p$ be a prime which divides $\alpha\beta$ and let $\mathfrak{p}$ be a prime ideal divisor of $[p]$ in $\mathcal O_\mathbb K$ with ramification index $e_{\mathfrak{p}}$. Then $\mathfrak{p}$ divides either $[\alpha]$ or $[\beta]$ and we shall assume, without loss of generality, that $\mathfrak{p}$ divides $[\alpha]$. Put $a'=(\beta-\alpha)a$ and $b'=(\beta-\alpha)b$. If $p^l$ exactly divides $[u_m]$ then $\mathfrak{p}^{e_{\mathfrak{p}}l}$ exactly divides $[b']$ for $m$ sufficiently large. Thus
 $$
 |u_m|_p \geq |a'b'|_p,
 $$
 and so
\begin{equation} \label{eq47}
\prod_{{\substack{p<T \\ p|\alpha\beta}}}|S(p)|^{-1}_p \leq \prod_{{\substack{p<T \\ p|\alpha\beta}}} |a'b'|^{-n}_p .
\end{equation}

Assume now that $p$ does not divide $\alpha\beta$ and let $t_n$, as in \eqref{eq10}, be the $n$-th term of the Lucas sequence associated with $(u_n)^\infty_{n=0}$. For positive integers $m$ and $r$ with $m\geq r$,
\begin{equation} \label{eq48}
u_m-\beta^ru_{m-r}=a'\alpha^{m-r}t_r.
\end{equation}
On setting $m=m(p)$ in \eqref{eq48} and letting $r$ run over those integers such that $m(p)-r\geq n$ we find that
\begin{equation} \label{eq49}
|u_{m(p)-1}\dots u_n|_p \geq \prod^{m(p)-n}_{r=1}(|t_r|_p|a'b'|_p).
\end{equation}
Let $l=l(p)$ be the smallest integer for which $p$ divides $t_l$; $l$ exists by Lemma \ref{lem4}. For any real number $x$ let $\lfloor x \rfloor$ denote the greatest integer less than or equal to $x$. By Lemma \ref{prop:2}, if $p>2$ then
\begin{equation} \label{eq50}
\prod^{m(p)-n}_{r=1}|t_r|_p = |t_l|^{s_1}_p|s_1!|_p,
\end{equation}
where $s_1=\lfloor \frac{m(p)-n}{l}\rfloor $, while if $p=2$
\begin{equation} \label{eq51}
\prod^{m(p)-n}_{r=1}|t_r|_2 = |t_l|^{s_1}_2|\frac{t_{2l}}{t_l}|^{s_2}_2|s_2!|_2,
\end{equation}
where $s_2=\lfloor \frac{m(p)-n}{2l}\rfloor $. Similarly on setting $m-r=m(p)$ in \eqref{eq48} and letting $r$ run over those integers such that $m(p)+r \leq 2n$ we find that for $p>2$
\begin{equation} \label{eq52}
|u_{m(p)+1}\dots u_{2n}|_p \geq |t_l|^{s_3}_{p}|s_3!|_p|a'b'|^{2n-m(p)}_p ,
\end{equation}
while for $p=2$,
\begin{equation} \label{eq53}
|u_{m(p)+1}\dots u_{2n}|_2 \geq |t_l|^{s_3}_{2}|\frac{t_{2l}}{t_l}|^{s_4}_2|s_4!|_2|a'b'|^{2n-m(p)}_2 ,
\end{equation}
where $s_3=\lfloor \frac{2n-m(p)}{l}\rfloor $ and $s_4=\lfloor \frac{2n-m(p)}{2l}\rfloor $. Thus, from \eqref{eq49}, \eqref{eq50} and \eqref{eq52} we see that if $p$ is a prime number which does not divide $2\alpha\beta$ then
\begin{equation} \label{eq54}
|S(p)|^{-1}_p \leq |t_l|^{-s}_p|s!|^{-1}_p|a'b'|^{-n}_p
\end{equation}
and
\begin{equation} \label{eq55}
|S(2)|^{-1}_2 \leq |t_l|^{-s}_2|\frac{t_{2l}}{t_l}|^{-s}_2|s!|^{-1}_2|a'b'|^{-n}_2
\end{equation}
where $s=\lfloor \frac{n}{l}\rfloor $.

By Lemma \ref{lem4} either $2$ divides $\alpha\beta$ or $2$ divides $t_n$ for some integer $n$ and $l(2)$ is either $2$ or $3$. But in the latter case, since $|t_l|\leq 2|\alpha|^l$,
\begin{equation} \label{eq56}
 |t_l|^{-s}_2|\frac{t_{2l}}{t_l}|^{-s}_2 \leq 2^n|\alpha|^{2n}.
\end{equation}
Therefore, by \eqref{eq47}, \eqref{eq54}, \eqref{eq55} and \eqref{eq56}
\begin{equation} \label{eq57}
\prod_{{\substack{p<T }}}|S(p)|^{-1}_p \leq  2^n|\alpha|^{2n}(\prod_{{\substack{p<T \\ p\nmid 2\alpha\beta}}}|t_l|^{-s}_p)n!|a'b'|^n.
\end{equation}

Now
\begin{equation} \label{eq58}
\prod_{{\substack{p<T \\ p\nmid 2\alpha\beta}}}|t_l|^{-s}_p = AB
\end{equation}
where
$$
A=\prod_{{\substack{l(p)< n/\log n \\ p<T \\ p\nmid 2\alpha\beta}}} |t_{l(p)}|^{-\lfloor \frac{n}{l(p)}\rfloor }_p
$$
and
$$
B=\prod_{{\substack{n/\log n \leq l(p) \\ p<T \\ p\nmid 2\alpha\beta}}} |t_{l(p)}|^{-\lfloor \frac{n}{l(p)}\rfloor }_p
$$
Observe that 
$$
A \leq \prod_{1\leq l < n/\log n}|t_l|^{\frac{n}{l}}
$$
and so
$$
A\leq  \prod_{1\leq l < n/\log n}2^n|\alpha|^{n}
$$
hence
\begin{equation} \label{eq59}
\log A \leq c_3n^2/\log n.
\end{equation}
Further when $l(p) \geq n/\log n$ we have
$$
\lfloor \frac{n}{l(p)}\rfloor  \leq \log n
$$
and, by Lemma \ref{lem4}, when $p<T$ we see that $l(p) < T+1$. Since 
$$
p+1 \geq l(p) \geq n/\log n
$$
it follows from Lemma \ref{lem8} that
$$
\log B \leq \pi(T)\log n(T+1)\exp(-\log (T+1)/51.9\log \log (T+1))\log (T+1)\log |\alpha|\log 2n
$$
hence, by \eqref{eq41},
\begin{equation} \label{eq60}
\log B \leq n^2\exp(-\log n/40,000\log \log n),
\end{equation}
for $n$ sufficiently large. By \eqref{eq57}, \eqref{eq58}, \eqref{eq59} and \eqref{eq60}
\begin{equation} \label{eq61}
\log \prod_{p<T}|S(p)|^{-1}_p \leq c_4 n\log n + c_3n^2/\log n + n^2\exp(-\log n/40,000\log \log n),
\end{equation}
for $n$ sufficiently large.

But the lower bound \eqref{eq42} for $\log |\prod_{\substack{m\in M}}u_m| $ is incompatible with the upper bound which follows from \eqref{eq43},  \eqref{eq46} and  \eqref{eq61}
for $n$ sufficiently large. This contradiction establishes our result.

\section{Acknowledgements}

This research was supported in part by the Canada Research Chairs Program and by grant A3528 from the Natural Sciences and Engineering Research Council of Canada.

\end{document}